\documentclass[12pt, reqno]{amsart}
\usepackage{amsmath}
\usepackage{amssymb}
\usepackage{hyperref}
\usepackage{amsfonts}
\usepackage{color}
\usepackage{enumerate}
\usepackage{mathtools}
\usepackage{url}
\usepackage{avant}
\usepackage[T1]{fontenc}
\usepackage[utf8]{inputenc}

\usepackage{lmodern}
\allowdisplaybreaks
\newtheorem{theorem}{Theorem}[section]
\newtheorem{corollary}[theorem]{Corollary}
\newtheorem{lemma}[theorem]{Lemma}

\theoremstyle{definition}

\theoremstyle{remark}
\newtheorem{remark}[theorem]{\bf{Remark}}
\numberwithin{equation}{section}
\theoremstyle{remark}

\hypersetup{
    colorlinks=true,
    linkcolor=blue,
    citecolor=red
    }

\newcommand{\beas}{\begin{eqnarray*}}
\newcommand{\eeas}{\end{eqnarray*}}
\newcommand{\bes} {\begin{equation*}}
\newcommand{\ees} {\end{equation*}}
\newcommand{\be} {\begin{equation}}
\newcommand{\ee} {\end{equation}}
\newcommand{\bea} {\begin{eqnarray}}
\newcommand{\eea} {\end{eqnarray}}
\newcommand{\ra} {\rightarrow}

\newcommand{\R}{\mathbb R}

\newcommand{\C}{\mathbb C}

\newcommand{\N}{\mathbb N}
\newcommand{\Z}{\mathbb Z}
\newcommand{\la}{\lambda}

\renewcommand{\Im}{\operatorname{Im}}

\title[ Unique Continuation Inequalities] { Unique Continuation Inequalities for the Schr\"odinger equations associated with the Special Hermite operators}
\author[J. Sarkar]{Jayanta Sarkar}
\address{Jayanta Sarkar \endgraf Department of Mathematics and Statistics, \endgraf Indian Institute of Science Education and Research Kolkata, \endgraf Mohanpur-741246, Nadia, West Bengal, India.} 
\email{jayantasarkarmath@gmail.com, jayantasarkar@iiserkol.ac.in}
\begin{document}
\maketitle
\begin{abstract}
We investigate unique continuation inequalities for solutions of the Schr\"odinger equations associated with special Hermite operators. Our main result establishes that if the solution remains small at two distinct time points outside sets of finite measure, then the solution also remains small throughout the entire space. We also explore analogous results for the Hermite-Schr\"odinger equations.  
\end{abstract}

\section{Introduction and Main result}
Let us consider the following initial value problem for the free Schr\"odinger equation on the Euclidean spaces
\begin{equation}\label{schrorn}
\begin{cases}
i\partial_t u(x, t) &= -\Delta_{\R^d} u(x, t), \quad (x, t)\in \R^d \times (0, \infty), \\
u(x, 0)&= u_0(x), \:\: u_0\in L^2(\R^d),
\end{cases}
\end{equation}
where $\Delta_{\R^d}$ is the Laplacian on $\R^d$, $d\geq 1$. Several results regarding the unique continuation of the equation \eqref{schrorn} can be proved using the following well-known identity 
\be \label{identity}
u(x, t)= (2it)^{-n/2}~e^{i|x|^2/(4t)}~ \mathcal F \left({e^{i|\cdot|^2/(4t)}u_0}\right)\left(\frac{x}{2t}\right),\quad (x,t)\in\R^d\times(0,\infty),
\ee
and the uncertainty principles for the Fourier transform on $\R^d$. Here, $\mathcal Fg$ denotes the Fourier transform of a function $g\in L^1(\R^d)\cap L^2(\R^d)$, which extends to all of $L^2(\R^d)$ as an isometry in the usual way. One such example of this, in recent times, is the following observability inequality at two points in time due to Wang, Wang, and Zhang \cite{WWZ}: there exists a constant $c=c_d>0$ such that for all $t>0$, all $r_1,r_2>0$, and all $u$ solving \eqref{schrorn},
\be \label{wwzo}
\int_{\R^d} |u_0(x)|^2\,dx\leq ce^{c r_1 r_2/t} \left(\int_{|x|>r_1} |u_0(x)|^2\,dx+ \int_{|x|>r_2} |u(x, t)|^2\,dx \right).
\ee
The proof of \eqref{wwzo} uses the following uncertainty principle due to Nazarov and Jaming \cite{Naz, Jamn}:
there exists a constant $C=C_d>0$ such that, for each pair of measurable sets $A$, $B\subset\R^d$ of finite Lebesgue
measure and for every $f\in L^2(\R^d)$,
\begin{equation}\label{naza}
    \int_{\R^d}|f(x)|^2\,dx\leq Ce^{C|A||B|}\left(\int_{\R^d\setminus A}|f(x)|^2\,dx+\int_{\R^d\setminus B}|\mathcal F f(\xi)|^2\:d\xi\right),
\end{equation}
where $|S|$ denotes the Lebesgue measure of a measurable set $S\subset\R^d$. Here and in what follows, whenever we talk about measurable subsets of some measurable space $X$ we always assume that they are of positive measure. 

In fact, it is shown in \cite{WWZ} that \eqref{naza} is essentially equivalent to the following statement: for each pair of measurable sets $A,\,B$ in $\R^d$ of finite measure and for $t>0$, there exists a constant $C=C(t,A,B)>0$ such that 
\be \label{wwzog}
\int_{\R^d} |u_0(x)|^2\,dx\leq C \left(\int_{\R^d\setminus A} |u_0(x)|^2\,dx+ \int_{\R^d\setminus B} |u(x, t)|^2\,dx \right),
\ee
for all solutions $u$ of \eqref{schrorn}. Several extensions of the unique continuation inequality, similar to \eqref{wwzo}, have been developed for various types of Schr\"odinger equations \cite{HS, WLHnon, WWKdV} on Euclidean spaces. Very recently, an analogue of \eqref{wwzog} has been proved in \cite{BR} for the free Schr\"odinger equation on Riemannian symmetric spaces of noncompact type.

Our aim in this article is to enquire whether an analogue of \eqref{wwzog} holds for the initial value problem \eqref{schrorn} by replacing $-\Delta_{\R^d}$ with some other elliptic operators. We first consider the scaled Hermite operator $$H(\la):=-\Delta_{\R^d}+\la^2\|x\|^2,$$
where $\la\in\R^*:=\R\setminus\{0\}$ and $\|\cdot\|$ is the  Euclidean norm. For a fixed $\la\in\R^*$, the corresponding Schr\"odinger equation (also known as Hermite-Schr\"odinger equation) is given by
\be\label{hermitesc}
\begin{cases}
i\partial_tu(x,t)&=H(\la)u(x,t),\quad(x,t)\in\R^d\times(0,\infty),\\
u(x,0)&=u_0(x),\quad u_0\in L^2(\R^d). 
\end{cases}
\ee
It can be shown that the solution $u(x,t)=e^{itH(\la)}u_0(x)$ is given by the fractional Fourier transform of $f$, whenever $2\la t\in\R\setminus\pi\Z$. Following the strategy used in \cite{Jamfr}, we will prove the following result.
\begin{theorem}\label{herobserv}
 Suppose $A$ and $B$ are measurable subsets of $\R^d$ with finite Lebesgue measure. Suppose $t\in (0,\infty)$ is such that $2\la t\in\R\setminus\pi\Z$. Then, for all solutions $u$ of the equation \eqref{hermitesc}, the following inequality holds.
\be \label{herobineq}
\int_{\R^d}|u_0(x)|^2\,dx\leq Ce^{C|\sin(2|\la|t)|^{-d}|\la|^d|A||B|}  \left(\int_{\R^d\setminus A} |u_0(x)|^2\,dx+ \int_{\R^d\setminus B}|u(x, t)|^2\,dx\right),
\ee 
where $C$ is the constant appearing in \eqref{naza}.\end{theorem} 
A version of the observability inequality for equation \eqref{hermitesc} with $\lambda=1$ has been proved in the recent work \cite[Theorem 1.3]{HWW1}. The difference between the inequality proved there with \eqref{herobineq} is that in \cite{HWW1}, the right-hand side is a double integration taken over an interval in the time variable together with the complement of a ball in the space variable.

\vspace{0.3cm}
We now focus on a family of elliptic operators on $\C^d$ namely the special Hermite operators $\{L_{\la}:\la\in\R^*\}$. We fix $\la\in\R^*$ and consider the following Schr\"odinger equation
\be\label{sphermitesc}
\begin{cases}
i\partial_tu(z,t)&=L_{\la}u(z,t),\quad(z,t)\in\C^d\times(0,\infty),\\
u(z,0)&=u_0(z),\quad u_0\in L^2(\C^d). 
\end{cases} 
\ee
The main result of this paper is the following analogue of \eqref{wwzog} for solutions of \eqref{sphermitesc}.
\begin{theorem}\label{spherobserv}
 Suppose $E$ and $\Omega$ are measurable subsets of $\C^d$ with finite Lebesgue measure. Suppose that $t\in (0,\infty)$ is such that $\la t\in\R\setminus\pi\Z$. Then there exists a positive constant $C(\la, t,  E, \Omega)$ such that the following inequality holds for all solutions $u$ of the equation \eqref{sphermitesc}
\be \label{spherobineq}
\int_{\C^d}|u_0(z)|^2\,dz\leq C(\la,t, E,  \Omega) \left(\int_{\C^d\setminus E} |u_0(z)|^2\,dz+ \int_{\C^d\setminus\Omega}|u(z, t)|^2\,dz\right).
\ee  
\end{theorem}
In contrast with the proof of \eqref{wwzog}, there is no analogue of the identity \eqref{identity} for solutions of \eqref{sphermitesc}. Thus, the inequality \eqref{spherobineq} cannot be reduced to any familiar uncertainty principle for the Fourier transform. To overcome this difficulty we employ ideas based on the proof of the uncertainty principle due to Amerin and Berthier \cite{AB} (see also \cite{GJ,WWKdV}).

\vspace{0.3cm}
As an immediate corollary of Theorem \ref{spherobserv}, we obtain the following result, which can be regarded as an analogue of Benedicks's theorem \cite{Be} for the solutions of the free Schr\"odinger equation associated with special Hermite operators.
\begin{corollary}
  Let $E$, $\Omega$ be measurable subsets of $\C^d$ with finite measure and let $u$ be a solution of \eqref{sphermitesc}. If there exists $t\in (0,\infty)$ such that $\la t\in\R\setminus\pi\Z$, and the supports of $u_0$ and $u(\cdot, t)$ are both contained within $E$ and $\Omega$, respectively, then it follows that $u$ vanishes identically.  
\end{corollary}

 \vspace{0.3cm}
 \textbf{Plan of the paper.} In Section 2, we provide essential preliminaries concerning the scaled Hermite operators, fractional Fourier transform, and the special Hermite operators. Additionally, within this section, we present the proof of Theorem \ref{herobserv} along with the necessary lemmas required for the proof of Theorem \ref{spherobserv}. The complete proof of Theorem \ref{spherobserv} is presented in the final section.

\vspace{0.3cm}
\textbf{Notations.} The letters $\N$, $\Z$, $\R$, and $\C$ will respectively denote the set of all natural numbers, the ring of integers, and the fields of real and complex numbers. For $ z \in \C $, we use the notation $\Im z$ for the imaginary part of $z$. We shall follow the standard practice of using the letters $c$, $C$, etc., for positive constants, whose value may change from one line to another. We shall also use $C(\varepsilon)$ or $C_{\varepsilon}$ to show their dependencies on the parameter $\varepsilon$. For a set $S$, we denote by $1_S$ its characteristic function and by $S^c$ its complement. 

\section{Preliminaries}
\subsection{Hermite operator and Fractional Fourier transform}
We begin with some preliminaries concerning the scaled Hermite operators and the fractional Fourier transform. We recall that the spectral decomposition of the scaled Hermite operator $H(\la)$ on $\R^d$ is given by
\begin{equation}\label{spectral}
 H(\la)=\sum_{k=0}^{\infty}(2k+d)|\la|P_k(\la),
\end{equation}
where $P_k(\la)$ is the projection on the eigenspace corresponding to the eigenvalue $(2k+d)|\la|$. For $g\in L^2(\R^d)$, $P_k(\la)g$ has the following expression
$$P_k(\la)g=\sum_{|\nu|=k}\langle g,\Phi_{\nu}^{\la}\rangle\,\Phi_{\nu}^{\la},$$
where $\nu=(\nu_1,\cdots,\nu_d)\in(\N\cup\{0\})^d$, $|\nu|=\sum_{i=1}^d\nu_i$, $\langle\cdot,\cdot\rangle$ is the standard inner product on $L^2(\R^d)$, and $\Phi_{\nu}^{\la}$ is the scaled Hermite functions defined as 
\begin{equation}\label{scaled}
    \Phi_{\nu}^{\la}(x)=|\la|^{\frac{d}{4}}\Phi_{\nu}(|\la|^{\frac{1}{2}}x),\quad x\in\R^d.
\end{equation}
For $\nu\in(\N\cup\{0\})^d$, the Hermite function $\Phi_{\nu}$ is obtained by taking
tensor products of one-dimensional Hermite functions. It is well known that for each $\la\in\R^*$, $\{\Phi_{\nu}^{\la}:\nu\in (\N\cup\{0\})^d\}$ is an orthonormal basis of $L^2(\R^d)$. We refer the reader to \cite{Tha} for more details on Hermite and special Hermite operators.

\vspace{0.3cm}
For $g\in L^1(\R^d)$ and $\alpha\in\R\setminus\pi\Z$, the fractional Fourier transform of order $\alpha$ is defined as \cite[p.422]{Jamfr1} (see also \cite{Jamfr})
\begin{equation}\label{frac}
\mathcal{F}_{\alpha}[g](\xi)=c_{\alpha}^d\gamma_{\alpha}(\xi)\mathcal{F}[\gamma_{\alpha}g](\xi/\sin\alpha),\quad\xi\in\R^d,
\end{equation}
where $\mathcal{F}[h]$ is the standard Fourier transform of an integrable function $h$ on $\R^d$, and
\begin{enumerate}
    \item [(i)] $c_{\alpha}=\frac{e^{i/2(\alpha-\pi/2)}}{|\sin\alpha|^{1/2}}$, and so $|c_{\alpha}|=|\sin\alpha|^{-1/2}$;
    \item [(ii)] $\gamma_{\alpha}(x)=e^{-i\pi\|x\|^2\cot\alpha}$, and so $|\gamma_{\alpha}(x)|=1$, for all $x\in\R^d$.
\end{enumerate}

\vspace{0.3cm}
It is known that $\mathcal{F}_{\alpha}$ extends from $L^1(\R^d)\cap L^2(\R^d)$ to $L^2(\R^d)$ as an unitary operator on $L^2(\R^d)$. It was noted in \cite{Jamfr, Jamfr1} that the fractional Fourier transform may alternatively be defined as 
\begin{equation}\label{frher}
 \mathcal{F}_{\alpha}[f]=\sum_{k=0}^{\infty}e^{-ik\alpha}\sum_{|\nu|=k}\langle f,\Phi_{\nu}\rangle\,\Phi_{\nu}, 
\end{equation}
where $f\in L^2(\R^d)$. For $k\in\Z$, we define $ \mathcal{F}_{2k\pi}f=f$, and $ \mathcal{F}_{(2k+1)\pi}f(\xi)=f(-\xi)$. From these definitions, we see that $\|\mathcal{F}_{\alpha}[f]\|_{L^2(\R^d)}=\|f\|_{L^2(\R^d)}$, 
$\mathcal{F}_{\alpha}\mathcal{F}_{\alpha'}=\mathcal{F}_{\alpha+\alpha'}$.
The following lemma is an analogue of \eqref{naza} for the fractional Fourier transform, which can be established by emulating the proof presented in \cite[Example 2.5]{Jamfr}, where it was demonstrated for the case of $d=1$. 
\begin{lemma}\label{nazafrac}
Suppose $\alpha,\,\beta$ are two real numbers such that $\delta:=\beta-\alpha\notin\pi\Z$ and $A_{\alpha},A_{\beta}$ are two sets of finite measure in $\R^d$. Then the following holds true for every $f\in L^2(\R^d)$
\begin{equation}\label{fracnaz}
\int_{\R^d}|f(x)|^2\,dx\leq Ce^{C|\sin\delta|^{-d}|A_{\alpha}||A_{\beta}|}\left(\int_{\R^d\setminus A_{\alpha}}|\mathcal{F}_{\alpha}[f](x)|^2\,dx+\int_{\R^d\setminus A_{\beta}}|\mathcal{F}_{\beta}[f](x)|^2\,dx\right),
\end{equation}
where $C$ is the constant appearing in \eqref{naza}.
\end{lemma}
\begin{proof}
    We set $\varphi=\gamma_{\delta}\mathcal{F}_{\alpha}[f]$. Then using the definition \eqref{frac}, we obtain 
\begin{align*}
&|\varphi(\xi)|=|\mathcal{F}_{\alpha}[f](\xi)|\\
    &|\mathcal{F}_{\beta}[f](\xi)|=|\mathcal{F}_{\beta-\alpha}\left[\mathcal{F}_{\alpha}[f]\right](\xi)|=|\sin\delta|^{-d/2}|\mathcal{F}[\varphi](\xi/\sin\delta)|,\quad\xi\in\R^d.
\end{align*}
Consequently, a straightforward computation yields 
\begin{equation}\label{fttt}
\begin{rcases}
    &\|\mathcal{F}_{\alpha}[f]\|_{L^2(\R^d)}=\|\varphi\|_{L^2(\R^d)},\\
    &\|\mathcal{F}_{\alpha}[f]\|_{L^2(\R^d\setminus A_{\alpha})}=\|\varphi\|_{L^2(\R^d\setminus A_{\alpha})},\\
    &\|\mathcal{F}_{\beta}[f]\|_{L^2(\R^d\setminus A_{\beta})}=\|\mathcal{F}[\varphi]\|_{L^2(\R^d\setminus (\sin\delta)^{-1}A_{\beta})},
\end{rcases}
    \end{equation}
where for $\rho\in\R^*$ and for a set $V\subseteq\R^d$, $\rho V:=\{\rho v:v\in V\}$. We apply Nazarov's uncertainty principle \eqref{naza} for the function $\varphi$, with $A=A_{\alpha}$, $B=(\sin\delta)^{-1}A_\beta$, and use \eqref{fttt} to get
\begin{align*}
 &\|f \|^2_{L^2(\R^d)}=\|\mathcal{F}_{\alpha}[f]\|^2_{L^2(\R^d)}=\|\varphi\|^2_{L^2(\R^d)}\\&\leq Ce^{C|A_{\alpha}||(\sin\delta)^{-1}A_{\beta}|}\left(\|\varphi\|^2_{L^2(\R^d\setminus A_{\alpha})}+\|\mathcal{F}[\varphi]\|^2_{L^2(\R^d\setminus (\sin\delta)^{-1}A_{\beta})}\right)\\
 &\leq Ce^{C|\sin\delta|^{-d}|A_{\alpha}||A_{\beta}|}\left(\|\mathcal{F}_{\alpha}[f]\|^2_{L^2(\R^d\setminus A_{\alpha})}+\|\mathcal{F}_{\beta}[f]\|^2_{L^2(\R^d\setminus A_{\beta})}\right).
\end{align*}
This completes the proof.
\end{proof}
We are now ready to prove Theorem \ref{herobserv}.
\begin{proof}[Proof of Theorem \ref{herobserv}]
 Using the spectral resolution of $H(\la)$, we can write the solution $u$ of \eqref{hermitesc} as follows:
\begin{equation}\label{solspec}
u(x,t)=e^{itH(\la)}u_0(x)=\sum_{k=0}^{\infty}e^{it(2k+d)|\la|}\sum_{|\nu|=k}\langle u_0,\Phi_{\nu}^{\la}\rangle\,\Phi_{\nu}^{\la}(x),\quad(x,t)\in\R^d\times(0,\infty).
\end{equation}
Setting $f_{\la}:=u_0\left(\frac{\cdot}{\sqrt{|\la|}}\right)$, and using the definition of scaled Hermite functions \eqref{scaled}, we get
$$\langle u_0,\Phi_{\nu}^{\la}\rangle=|\la|^{-\frac{d}{4}}\langle f_{\la},\Phi_{\nu}\rangle.$$
Using this observation in \eqref{solspec}, we see that
$$u(x,t)=\sum_{k=0}^{\infty}e^{it(2k+d)|\la|}\sum_{|\nu|=k}\langle f_{\la},\Phi_{\nu}\rangle\,\Phi_{\nu}(\sqrt{|\la|}\,x).$$
In view of the definition of the fractional Fourier transform \eqref{frac}, we can write from the above 
$$u(x,t)=e^{itd|\la|}\mathcal{F}_{-2|\la|t}[f_{\la}](\sqrt{|\la|}\,x).$$
Thus, a simple change of variable yields
\begin{equation}\label{substi1}
    |\la|^{\frac{d}{2}}\|u(\cdot,t)\|^2_{L^2(\R^d\setminus B)}=\|\mathcal{F}_{-2|\la|t}[f_{\la}]\|^2_{L^2(\R^d\setminus|\la|^{\frac{1}{2}}\,B)}.
\end{equation}
Similarly, We also note that \begin{equation}\label{substi2}
\|f_{\la}\|^2_{L^2(\R^d\setminus|\la|^{\frac{1}{2}}\,A)}=|\la|^{\frac{d}{2}}\|u_0\|^2_{L^2(\R^d\setminus A)},\quad \|f_{\la}\|^2_{L^2(\R^d)}=|\la|^{\frac{d}{2}}\|u_0\|^2_{L^2(\R^d)}.
\end{equation}
Finally, using the hypothesis that $2t|\lambda|\in\R\setminus\pi\mathbb{Z}$, we can apply Lemma \ref{nazafrac} for $f=f_{\la}$, with $$\alpha=0,\quad  \beta=-2|\la|t,\quad A_{\alpha}=|\la|^{\frac{1}{2}}\, A,\quad  A_{\beta}=|\la|^{\frac{1}{2}}\,B,$$ 
and then using \eqref{substi1}, \eqref{substi2}, we obtain
$$\int_{\R^d}|u_0(x)|^2\,dx\leq Ce^{C|\sin(2|\la|t)|^{-d}|\la|^d|A||B|} \left(\int_{\R^d\setminus A} |u_0(x)|^2\,dx+\int_{\R^d\setminus B}|u(x, t)|^2\,dx\right),$$
where $C$ is the constant appearing in \eqref{naza}.
\end{proof}

\vspace{0.3cm}
\subsection{Special Hermite Operator} For a nonzero real number $\la$, the special Hermite operator $L_{\lambda}$ is explicitly given by 
	$$L_\lambda = -\Delta_{\C^d}+\frac{1}{4} \lambda^2 \|z\|^2- i \lambda\sum_{j=1}^{d}\left(x_j\frac{\partial}{\partial y_j}-y_j\frac{\partial}{\partial x_j}\right),$$
	where $z=(z_1,\cdots,z_d)=(x_1+iy_1,\cdots,x_d+iy_d)\in\C^d$, $\|z\|$ is the usual Euclidean norm of $z\in\C^d$, and $\Delta_{\C^d}$ is the Laplacian on $\C^d$.
	
We recall the Schrödinger equation \eqref{sphermitesc} associated with the operator $L_{\lambda}$, is given by
\begin{equation*}
    \begin{cases}
        i\partial_tu(z,t)+L_{\lambda}u(z,t)&=0,\quad (z,t)\in\C^d\times(0,\infty),\\
u(z,0)&=u_0(z),\quad u_0\in L^2(\C^d).
    \end{cases}
\end{equation*}
It is known that the solution $u(z,t)$ of \eqref{sphermitesc} is provided by the following formula \cite{PRT}:
\begin{equation}\label{ssol}
u(z,t)=e^{itL_{\la}}f(z)=f\times_{\la}p^{\la}_{it}(z),\quad(z,t)\in\C^d\times(0,\infty),
\end{equation}
where the kernel $p^{\la}_{it}$ of the Schr\"odinger semigroup $e^{itL_{\la}}$ is given by
\begin{equation}\label{skernel}
p^{\la}_t(z)=c_d\left(\frac{\la}{\sin\la t}\right)^de^{\frac{i\la}{4}\:\|z\|^2\cot t\la}
\end{equation}
and $g\times_{\la}h$ denotes the $\la$-twisted convolution of two suitable functions $g,\,h$ on $\C^d$, which is defined as
\begin{equation}\label{twistedconv}
g\times_{\la}h(z)= \int_{\C^d} g(z-w) h(w) e^{\frac{i\la}{2} \Im (z \cdot \bar{w})}\,dw.
\end{equation}
We note from (\ref{skernel}) that $e^{itL_{\la}}$ is periodic in $t$ and $p^{\la}_{it}$ exists as long as $\la t$ is not an integral multiple of $\pi$. Also, $\{e^{i\tau L_{\la}}:\tau\in\R\}$ is a one-parameter group of unitary
operators on $L^2(\C^d)$ and we denote $S_{\tau}^{\la}:=e^{i\tau L_{\la}}$, $\tau\in\R$. Moreover, the solution $u(z,t)$ of \eqref{sphermitesc} is $\frac{2\pi}{|\la|}$-periodic in $t$.

\vspace{0.3cm}
An interesting feature of the special Hermite operator $L_{\la}$ is that it is invariant under a certain translation on $\C^d$, namely the $\la$-twisted translation. For $w\in\C^d$, the $\la$-twisted translation of a function $g$ defined on $\C^d$ by $w$ is defined as follows:
\begin{equation}\label{lamtrans}
T^{\la}_wg(z)=e^{\frac{i\la}{2} \Im (w \cdot \bar{z})}g(z-w),\:\:\:z\in\C^d.
\end{equation}
One can deduce the following two observations regarding $\la$-twisted convolution from the definitions (\ref{twistedconv}) and \eqref{lamtrans}:
\begin{equation}\label{twi}
    \begin{rcases}
        &g\times_{\la}h=h\times_{-\la}g,\\
    &(T^{\la}_wg)\times_{\la}h=T^{\la}_w(g\times_{\la}h).
    \end{rcases}
\end{equation}
\subsection{Some Auxilary Lemmas} It is evident from (\ref{skernel}) that  $p^{\la}_{it}=p^{-\la}_{it}$, whenever $t\la\in\R\setminus\pi\Z$. Thus, from now on, we fix some $\la>0$ and $t>0$ such that $\la t\in\R\setminus\pi\Z$.  We fix two measurable sets $E,\:\Omega$ in $\C^d$ of finite measure and define an operator $P^{\la}_t$ on $L^2(\C^d)$ as follows:
\be \label{P-defn}
P^{\la}_tf:=P^{\la}_{E, \Omega, t}f:=1_\Omega S^{\la}_t 1_E f, \quad f\in L^2(\C^d).
\ee
As $S^{\la}_t=e^{itL_{\la}}$ is an unitary operator on $L^2(\C^d)$, $P^{\la}_t$ is a bounded linear operator on $L^2(\C^d)$ with norm at most one. The subsequent lemma presents a criterion for the validity of our main theorem, Theorem \ref{spherobserv}.
\begin{lemma}\label{mainlemma}
If $\|P^{\la}_t\|_{L^2(\C^d)\to L^2(\C^d)}<1$, then there is a constant $C=C(t,\la,E,\Omega)>0$ such that 
\begin{equation}\label{mainineqlemma}
\int_{\C^d}|f(z)|^2\:dz\leq C\left(\int_{E^c}|f(z)|^2\:dz+\int_{\Omega^c}|S^{\la}_tf(z)|^2\:dz\right),
\end{equation}
for any $f\in L^2(\C^d)$.
\end{lemma}
\begin{proof}
We first assume that $1_Ef=f$, with $f\in L^2(\C^d)$. Then $P_t^{\la}f=1_{\Omega}S_t^{\la}f$ and hence
$$\|1_\Omega S^{\la}_t f\|_{L^2(\C^d)}=\|P_t^{\la} f\|_{L^2(\C^d)}\leq \|P_t^{\la}\|_{L^2(\C^d) \ra L^2(\C^d)} \|f\|_{L^2(\C^d)},$$
which implies
\begin{eqnarray}
\|1_{\Omega^c} S^{\la}_t f\|_{L^2(\C^d)}&\geq&\|S^{\la}_t f\|_{L^2(\C^d)}-\|1_\Omega S^{\la}_t f\|_{L^2(\C^d)}\nonumber\\&\geq& \|S^{\la}_t f\|_{L^2(\C^d)}-\|P_t^{\la}\|_{L^2(\C^d) \ra L^2(\C^d)} \|f\|_{L^2(\C^d)}\nonumber\\
&=&\|f\|_{L^2(\C^d)}-\|P_t^{\la}\|_{L^2(\C^d) \ra L^2(\C^d)} \|f\|_{L^2(\C^d)}\nonumber\\&=&(1-\|P_t^{\la}\|_{L^2(\C^d) \ra L^2(\C^d)})\|f\|_{L^2(\C^d)}\label{specialcase}.
\end{eqnarray}
We set $C_0=(1-\|P_t^{\la}\|_{L^2(\C^d) \ra L^2(\C^d)})$. Since the hypothesis states that $\|P^{\lambda}_t\|_{L^2(\mathbb{C}^d)\to L^2(\mathbb{C}^d)}<1$, \eqref{specialcase} shows that (\ref{mainineqlemma}) holds with the constant $C=C_0^{-2}$ whenever $1_Ef=f$, with $f\in L^2(\mathbb{C}^d)$.

Now, for an arbitrary $f\in L^2(\C^d)$, we note that
\begin{eqnarray*}
\|f\|_{L^2(\C^d)}&\leq&\|1_Ef\|_{L^2(\C^d)}+\|1_{E^c}f\|_{L^2(\C^d)}\\&\leq&C_0^{-1}\|1_{\Omega^c} S^{\la}_t 1_Ef\|_{L^2(\C^d)}+\|1_{E^c}f\|_{L^2(\C^d)}\hspace{0.6cm}(\text{using (\ref{specialcase})})\\&=&C_0^{-1}\|1_{\Omega^c} S^{\la}_t (f-1_{E^c}f)\|_{L^2(\C^d)}+\|1_{E^c}f\|_{L^2(\C^d)}\\&\leq&C_0^{-1}\|1_{\Omega^c} S^{\la}_t f\|_{L^2(\C^d)}+C_0^{-1}\|1_{\Omega^c} S^{\la}_t 1_{E^c}f\|_{L^2(\C^d)}+\|1_{E^c}f\|_{L^2(\C^d)}\\&\leq&C_0^{-1}\|1_{\Omega^c} S^{\la}_t f\|_{L^2(\C^d)}+(C_0^{-1}+1)\|1_{E^c}f\|_{L^2(\C^d)}\quad(\text{as $S^{\la}_t$ is an isometry})
\end{eqnarray*}
Thus, by taking $C(t,\la,E,\Omega)=2(C_0^{-1}+1)^2$, we obtain (\ref{mainineqlemma}).
\end{proof}
\begin{lemma}\label{ptcom}
The operator $P_t^{\la}$ is compact.
\end{lemma}
\begin{proof}
Using the definitions (\ref{P-defn}) of $P_t^{\la}$ and that of $S^{\la}_t=e^{itL_{\la}}$ (see \eqref{ssol}), we write
\begin{equation*}
P_t^{\la}f(z)=1_{\Omega}(z)S_t^{\la}(1_Ef)(z)=1_{\Omega}(z)(1_Ef)\times_{\la}\,p_{it}^{\la}(z)=1_{\Omega}(z)\,p_{it}^{\la}\times_{-\la}(1_Ef)(z), 
\end{equation*}
where we have used \eqref{twi} in the last equality. Using the definition of $\la$-twisted convolution \eqref{twistedconv}, we obtain
$$P_t^{\la}f(z)=1_{\Omega}(z)\int_{\C^d}1_{E}(w)f(w)\,p_{it}^{\la}(z-w)e^{-\frac{i\la}{2} \Im (z \cdot \bar{w})}\:dw=\int_{\C^d}K_t(z,w)f(w)\:dw,$$
where$$K_t(z,w)=1_{\Omega}(z)1_{E}(w)\,p_{it}^{\la}(z-w)e^{-\frac{i\la}{2} \Im (z \cdot \bar{w})}.$$
By plugging in the expression of $p_{it}^{\lambda}$ \eqref{skernel}, we arrive at:
\begin{eqnarray*}
\int_{\C^d}\int_{\C^d}|K_t(z,w)|^2\:dw\:dz&=&\int_{\C^d}1_{\Omega}(z)\int_{\C^d}1_{E}(w)|p_{it}^{\la}(z-w)|^2\:dw\:dz\nonumber\\&=&\int_{\C^d}1_{\Omega}(z)\int_{\C^d}\left|\frac{\la}{\sin\la t}\right|^{2d}1_E(w)\:dw\:dz\nonumber\\&=&|\Omega||E|\left|\frac{\la}{\sin\la t}\right|^{2d}<\infty,\quad\text{as}\,\,\la t\in\R\setminus\pi\Z.
\end{eqnarray*}
This shows that $P^{\la}_t$ is a Hilbert-Schimdt operator on $L^2(\C^d)$ and hence $P^{\la}_t$ is compact.
\end{proof}
\begin{lemma}\label{support}
Suppose for some $f\in L^2(\C^d)$, we have the following equality $$\|P_t^{\la}f\|_{L^2(\C^d)}=\|f\|_{L^2(\C^d)}.$$ 
Then $f=0$ a.e. on $E^c$ and $S_t^{\la}f=0$ a.e. on $\Omega^c$.
\end{lemma}
\begin{proof}
If $f$ would not vanishes almost everywhere on $E^c$, then $\|f\|_{L^2(E^c)}$ is nonzero. However, our hypothesis on $f$ yields the following
$$\|f\|_{L^2(\C^d)}=\|1_\Omega S^{\la}_t 1_E f\|_{L^2(\C^d)}\leq \|S^{\la}_t 1_E f\|_{L^2(\C^d)}\leq\|1_E f\|_{L^2(\C^d)}.$$
This leads to a contradiction. Hence, we must have $f=1_Ef$ a.e., which consequently implies 
$$\|1_{\Omega} S_t^{\la}f\|_{L^2(\C^d)}=\|1_{\Omega} S_t^{\la}(1_Ef)\|_{L^2(\C^d)}=\|P_t^{\la}f\|_{L^2(\C^d)}=\|f\|_{L^2(\C^d)}=\|S_t^{\la}f\|_{L^2(\C^d)}.$$ 
Therefore, $S_t^{\la}f=0$ a.e. on $\Omega^c$. 
\end{proof}
For $z\in\C^d$ and $A\subseteq\C^d$, we define $$z+A=\{z+w:w\in A\},\quad A^{-1}=\{-w:w\in A\}.$$
We note that for any pair of sets $U$, $V$ in $\C^d$ with finite Lebesgue measure, and for $z\in\C^d$
$$|U\cap(z+V)|=\int_{\C^d}1_U(z')1_{z+V}(z')\,dz'=\int_{\C^d}1_U(z')1_{V^{-1}}(z-z')\,dz'=1_U\ast1_{V^{-1}}(z),$$
where $\ast$ denotes the usual group convolution on $\C^d$. We now establish the following lemma, which will play a pivotal role in the proof of our main theorem.
\begin{lemma}\label{crucial}
 Let $A_0\subseteq A$, $B_0\subseteq B$ be four measurable sets in $\C^d$ with finite measure. Then for every $\epsilon>0$, sufficiently small, there exists $w_{\epsilon}\in\C^d$ such that 
 \begin{equation}\label{bothless}
     |A\cup (w_{\epsilon}+A_0)|\leq |A|+\epsilon,\quad |B\cup (w_{\epsilon}+B_0)|\leq |B|+\epsilon,
 \end{equation}
and one of the following holds
\begin{align}
    &|A|<|A\cup (w_{\epsilon}+A_0)|\label{eithera},\\
    &|B|<|B\cup (w_{\epsilon}+B_0)|\label{orb}.
\end{align}
\end{lemma}
\begin{proof}
  We first define the function $h_A:\C^d\to[|A|,\infty)$ as follows:
  $$h_A(z)=|A\cup(z+A_0)|,\,\,\,z\in\C^d.$$
 It is clear that $h_A$ has the following expression 
\begin{align}\label{union}
    h_A(z)=|A|+|A_0|-|A\cap(z+A_0)|=|A|+|A_0|-1_{A}\ast 1_{A_0^{-1}}(z),\,\,\,\text{for all}\,\,z\in\C^d.
\end{align}
Since $\C^d$ is unimodular and $|A|$, $|A_0|$ are finite, $1_{A}\ast 1_{A_0^{-1}}$ is continuous and  vanishes at infinity. Consequently, $h_A$ is continuous in $\C^d$ and there exists $w_0\in\C^d$ such that $$1_{A}\ast 1_{A_0^{-1}}(w_0)<\frac{|A_0|}{2}.$$ 
Applying the inequality above in \eqref{union}, we obtain $h_A(w_0)>|A|$. Then there exists $0<\alpha<|A_0|$ such that
$$h_A(w_0)=|A|+\alpha.$$
We fix $\epsilon\in(0,\alpha)$, and consider the level set
$$S_{\epsilon}=\{w\in\C^d:h_A(w)=|A|+\epsilon\}.$$
 We now show that $S_{\epsilon}$ is a non-empty compact set. To do this, we consider the continuous function $g:[0,1]\to\R$ defined by
 $$g(r)=h_A(rw_0),$$
where we note that $g(0)=|A|$ and $g(1)=|A|+\alpha$. Hence, by the intermediate value property, $S_{\epsilon}$ is non-empty. Since $S_{\epsilon}$ is a level set of a continuous function, it is closed. If there exists a sequence $\{w_j\}_{j=1}^{\infty}\subseteq S_{\epsilon}$ such that $\|w_j\|\to\infty$ as $j\to\infty$, then we can derive the following from \eqref{union}:
$$|A|+\epsilon=\lim_{j\to\infty}h_A(w_j)=|A|+|A_0|-\lim_{j\to\infty}1_{A}\ast 1_{A_0^{-1}}(w_j)=|A|+|A_0|,$$
which is a contradiction as $\epsilon<|A_0|$. Thus, $S_{\epsilon}$ is compact, and hence
$$\delta:=\min_{w\in S_{\epsilon}}\|w\|>0.$$
 Using the continuity of $h_A$ once again, we get
\begin{equation}\label{usedlater}
    |A|\leq h_A(z)\leq |A|+\epsilon,\quad\text{for all}\quad\|z\|\leq\delta.
\end{equation}
 If not, then we would have some $z_0\in\C^d$ with $\|z_0\|\leq\delta$ such that $h_A(z_0)> |A|+\epsilon$. Then we consider the function $$F(r)=h_A(rz_0),\quad r\in[0,1],$$
 which satisfies $F(0)=|A|$, $F(1)>|A|+\epsilon$. By the intermediate value property, we get $r_0\in(0,1)$ such that $$F(r_0)=h_A(r_0z_0)=|A|+\epsilon,$$
 and so $r_0z_0\in S_{\epsilon}$. But this contradicts the definition of $\delta$ as $\|r_0z_0\|<\|z_0\|\leq\delta$.

Similar to the function $h_A$, we define another continuous function $h_B$ in $\C^d$ via the following formula
 $$h_B(z)=|B\cup (z+B_0)|,\quad z\in\C^d.$$
 We now fix $w'\in S_{\epsilon}$ such that $\|w'\|=\delta$ and divide the discussions into two cases.
\vspace{0.3cm}
 
 \textbf{Case 1:} $h_B(w')\leq |B|+\epsilon$. This together with the fact $h_A(w')=|A|+\epsilon$ implies that \eqref{bothless} and \eqref{eithera} hold with $w_{\epsilon}=w'$.
\vspace{0.3cm}

 \textbf{Case 2:} $h_B(w')> |B|+\epsilon$. Then applying the intermediate value property to the continuous function $r\mapsto h_B(rw')$, in the interval $[0,1]$, we can find $r_0\in(0,1)$ such that $h_B(r_0w')= |B|+\epsilon$. Moreover, since $\|r_0w'\|<\|w'\|=\delta$, we deduce from \eqref{usedlater} that
 $$|A|\leq h_A(r_0w')\leq |A|+\epsilon.$$
 Thus, in this case, \eqref{bothless} and \eqref{orb} holds with $w_{\epsilon}=r_0w'$.
\end{proof}
\section{Proof of the main theorem}
At this point, we are ready to present the proof of our main theorem.
\begin{proof}[\textit{Proof of Theorem \ref{spherobserv}}] 
We recall the definition \eqref{P-defn} of $P_t^{\la}$. 
In view of Lemma \ref{mainlemma}, it suffices to prove that 
\begin{equation}\label{ptless1}
    \|P_t^{\la}\|_{L^2(\C^d)\to L^2(\C^d)}<1.
\end{equation}
As we have previously observed that
\begin{equation*}
    \|P_t^{\la}\|_{L^2(\C^d)\to L^2(\C^d)}\leq 1,
\end{equation*}
it is enough to show that 
\begin{equation*}
    \|P_t^{\la}\|_{L^2(\C^d)\to L^2(\C^d)}\neq 1.
\end{equation*}
If possible, let us assume that 
\begin{equation}\label{contra}
\|P_t^{\la}\|_{L^2(\C^d)\to L^2(\C^d)}=1.
\end{equation}
Since $P_t^{\la}$ is compact (see Lemma \ref{ptcom}), \eqref{contra} implies that
there exists $f_0\in L^2(\C^d)$ such that
$$\|P_t^{\la}f_0\|_{L^2(\C^d)}=\|f_0\|_{L^2(\C^d)}=1.$$
Lemma \ref{support} then shows that $f_0=0$, a.e on $E^c$ and $S_{t}^{\la}f_0=0$, a.e on $\Omega^c$. We thus obtain two measurable sets $E_0\subseteq E$, $\Omega_0\subseteq\Omega$ such that $|E_0|=|E|$, $|\Omega_0|=|\Omega|$ and
\begin{equation}\label{supprt}
    1_{E_0}f_0=f_0\quad \text{a.e.},\quad  1_{\Omega_0}S_{t}^{\la}f_0=S_{t}^{\la}f_0\quad\text{a.e}.
\end{equation}
We fix a large positive integer $N$. We will now inductively construct a sequence $\{w_j\}_{j=1}^{\infty}$ such that 
\begin{equation}\label{induction}
|E_j|\leq |E_{j-1}|+\frac{1}{2^{j+N}}, \:\:\:\:|\Omega_j| \leq |\Omega_{j-1}|+\frac{1}{2^{j+N}};
\end{equation}
and one of the following strict inequalities holds
\begin{equation}
\label{eqn-a}
|E_{j-1}| < |E_j|,\:\: \:\:|\Omega_{j-1}|< |\Omega_j|,
\end{equation}
where for $j=1,2,\cdots$, $E_j$ and $\Omega_j$ are measurable subsets of $\C^d$ defined as follows
\begin{equation}\label{ejdefn}
E_j=E_{j-1}\cup (w_j+E_0),\hspace{0.4cm}\Omega_j=\Omega_{j-1}\cup(w_j+\Omega_0).
\end{equation}
Applying Lemma \ref{crucial} for 
$$\epsilon=\frac{1}{2^{1+N}},\quad A=A_0=E_0,\quad B=B_0=\Omega_0,$$ we get $w_1\in\C^d$ such that if we define $E_1,\:\Omega_1$ according to (\ref{ejdefn}) then these measurable sets satisfy (\ref{induction}) and (\ref{eqn-a}) for $j=1$. Suppose we have constructed $\{w_j\}_{j=1}^m$, $\{E_j\}_{j=1}^m$ and $\{\Omega_j\}_{j=1}^m$ satisfying the relations (\ref{induction})-(\ref{ejdefn}) for $j=1,\cdots,m$. For the $(m+1)$th step, applying Lemma \ref{crucial} once again for $$\epsilon= \frac{1}{2^{m+1+N}},\quad A=E_m,\quad A_0=E_0,\quad B=\Omega_m,\quad B_0=\Omega_0,$$ we get $w_{m+1}\in\C^d$ such that the sets $E_{m+1}$, $\Omega_{m+1}$ defined according to (\ref{ejdefn}) satisfy the inequalities (\ref{induction}) and (\ref{eqn-a}). We define $$f_j=T^{\la}_{w_j}f_0,\quad j\in\N.$$
It follows from the definition of twisted translation (\ref{lamtrans}) and from \eqref{supprt} that $1_{w_j+E_0}f_j=f_j$. As $S_t^{\la}$ is given by the $\la$-twisted convolution with the kernel $p_{it}^{\la}$ and $\la$-twisted translation commutes with the $\la$-twisted convolution (see \eqref{twi}), it also follows that $$S_t^{\la}f_j=(T_{w_j}f_0)\times_{\la}p_{it}^{\la}=T_{w_j}(f_0\times_{\la}p_{it}^{\la})=T_{w_j}(S_t^{\la}f_0).$$
Hence, by \eqref{supprt}, $1_{w_j+\Omega_0}S_t^{\la}f_j=S_t^{\la}f_j$. Using our construction \eqref{ejdefn}, we observe that
\begin{equation}\label{subset}
\begin{rcases}
    &E_{j-1}\subset E_j,\hspace{0.7cm}E_j\setminus E_{j-1}\subset w_j+E_0\\
    &\Omega_{j-1}\subset\Omega_j,\hspace{0.7cm}\Omega_j\setminus \Omega_{j-1}\subset w_j+\Omega_0
    \end{rcases}
\quad j\in\N.
\end{equation}
We now claim that $\{f_j\}_{j=1}^{\infty}$ is linearly independent. To establish this, we take an arbitrary finite subcollection $\{f_{j_k}\}_{k=1}^m$, where $j_1<\cdots<j_m$. Since $1_{E_0}f_0=f_0$ a.e., and 
$$f_{j_k}(z)=e^{\frac{i\la}{2} \Im (w_{j_k}\cdot \bar{z})}f_0(z-w_{j_k}),$$
it follows from \eqref{subset} that $$f_{j_k}(z)\neq 0,\:\:\:\:\text{for almost every}\:\:z\in E_{j_k}\setminus E_{j_{k-1}}$$
We note that $f_{j_k}$ vanishes almost everywhere in $\C^d\setminus E_{j_k}$. According to (\ref{eqn-a}), either $|E_{j_m}\setminus E_{j_{m-1}}|>0$, or $|\Omega_{j_m}\setminus \Omega_{j_{m-1}}|>0$. If $|E_{j_m}\setminus E_{j_{m-1}}|>0$, then it follows right away that $f_{j_m}$ is not a linear combination of $f_{j_1},\cdots,f_{j_{m-1}}$ as they all vanishes almost everywhere in $\C^d\setminus E_{j_{m-1}}$. This shows that $\{f_{j_k}\}_{k=1}^m$ is linearly independent. On the other hand, if $|\Omega_{j_m}\setminus \Omega_{j_{m-1}}|>0$, then by similar argument we can prove that $\{S_t^{\la}f_{j_k}\}_{k=1}^m$ is linearly independent. But $S_t^{\la}$ is invertible. Hence, $\{f_{j_k}\}_{k=1}^m$ is linearly independent. This proves our claim.

From (\ref{induction}), we conclude that $$|E_j|\leq |E_0|+\sum_{k=1}^j\frac{1}{2^{k+N}};\quad|\Omega_j|\leq |\Omega_0|+\sum_{k=1}^j\frac{1}{2^{k+N}},$$ and hence
$$|\bigcup_{j=1}^{\infty}E_j|\leq E_0|+1;\quad|\bigcup_{j=1}^{\infty}\Omega_j|\leq|\Omega_0|+1.$$ 
We set \,$\mathcal{E}=\bigcup_{j=1}^{\infty}E_j,\,\,\mathcal{O}=\bigcup_{j=1}^{\infty}\Omega_j.$ Then, by Lemma \ref{ptcom}, the operator $\mathcal{P}:=1_{\mathcal{O}}S^{\la}_t1_{\mathcal{E}}$ is compact. Consequently, $S^{\la}_{-t}\mathcal{P}$ is also compact, as $S^{\la}_{-t}$ is an isometry. Employing the following established facts $$1_{\mathcal{E}}f_j=f_j, \quad 1_{\mathcal{O}}S^{\la}_tf_j=S^{\la}_tf_j,\quad\text{for all}\quad j\geq 0,$$ we get
$$S^{\la}_{-t}\mathcal{P}f_j=S^{\la}_{-t}1_{\mathcal{O}}S^{\la}_t1_{\mathcal{E}}f_j=S^{\la}_{-t}1_{\mathcal{O}}S^{\la}_tf_j=S^{\la}_{-t}S^{\la}_tf_j=f_j.$$
This shows that the compact operator $S^{\la}_{-t}\mathcal{P}$ possesses an infinite number of linearly independent eigenfunctions associated with the eigenvalue $1$. This leads to a contradiction, thereby invalidating our initial assumption (\ref{contra}). This completes the proof of Theorem \ref{spherobserv}.
\end{proof}
\begin{remark}
As we have used \eqref{naza} to prove Theorem \ref{herobserv}, the constant in the inequality \eqref{herobineq} explicitly reveals its dependence on $\lambda$, $t$, $A$, and $B$. However, this is not the case for Theorem \ref{spherobserv} since the proof relies on a contradiction argument.
\end{remark}

\section*{Acknowledgments}
The author is grateful to Mithun Bhowmik, Pritam Ganguly, and Swagato K. Ray for useful discussions. The author is supported by the INSPIRE faculty fellowship (Ref. no. DST/INSPIRE/04/2022/002544) from the Department of Science and Technology, Government of India.

\bibliographystyle{alphaurl}
\bibliography{specialhermite}
\end{document}